\documentclass[11pt,reqno]{amsart}
\usepackage{color}

\usepackage{amsfonts,amscd}
\usepackage{amssymb}
\usepackage{url}
\usepackage[english]{babel}

\theoremstyle{plain}
\newtheorem{theorem}                 {Theorem}      [section]
\newtheorem{conjecture}   [theorem]  {Conjecture}

\newtheorem{proposition}  [theorem]  {Proposition}

\theoremstyle{definition}

\newtheorem{definition}   [theorem]  {Definition}

\numberwithin{equation}{section}

\def \H{{\mathbb H}}

\def \rn{{\mathbb R}}

\def \F{\mathcal F}

\def \H{\mathcal H}

\def \V{\mathcal V}

\def\nab#1#2{\hbox{$\nabla$\kern -.3em\lower 1.0 ex
		\hbox{$#1$}\kern -.1 em {$#2$}}}

\def \g{\mathfrak{g}}

\def \k{\mathfrak{k}}

\def \m{\mathfrak{m}}

\def \SLR#1{\text{\bf SL}_{#1}(\rn)}
\def \SL2{\widetilde{\text{\bf SL}}_{2}(\rn)}
\def \slr#1{\mathfrak{sl}_{#1}(\rn)}

\def \SO#1{\text{\bf SO}(#1)}
\def \so#1{\mathfrak{so}(#1)}

\def \SU#1{\text{\bf SU}(#1)}
\def \su#1{\mathfrak{su}(#1)}

\def\jid(#1#2#3){\left[\left[#1,#2\right],#3\right] + \left[\left[#3,#1\right],#2\right] + \left[\left[#2,#3\right],#1\right]} 

\def\ss#1#2{_{#1#2}}

\DeclareMathOperator{\trace}{trace}

\numberwithin{equation}{section}
\allowdisplaybreaks

\def\ss#1#2{_{#1#2}}
\def\jid(#1#2#3){\left[\left[#1,#2\right],#3\right] + \left[\left[#3,#1\right],#2\right] + \left[\left[#2,#3\right],#1\right]} 
\def\BV(#1#2){B^{\V}(#1,#2)}
\def\hproj(#1#2){\frac{1}{2}\,\H\,(\nab(#1#2) + \nab(#2#1))}
\def\BH(#1#2){B^{\H}(#1,#2)}
\def\vproj(#1#2){\frac{1}{2}\,\V\,(\nab(#1#2) + \nab(#2#1))}
\newcommand{\scalar}[2]{\langle #1,#2 \rangle}
\def\lieb(#1#2){\left[#1,#2\right]}

\begin{document}

\title[Conformal foliations on Lie groups]{Conformal foliations on Lie groups\\ and complex-valued harmonic morphisms}


\author{Elsa Ghandour}
\address{Mathematics, Faculty of Science\\
University of Lund\\
Box 118, Lund 221 00\\
Sweden}
\email{Elsa.Ghandour@math.lu.se}

\author{Sigmundur Gudmundsson}
\address{Mathematics, Faculty of Science\\
	University of Lund\\
	Box 118, Lund 221 00\\
	Sweden}
\email{Sigmundur.Gudmundsson@math.lu.se}

\author{Thomas B. Turner}
\address{Mathematics, Faculty of Science\\
	University of Lund\\
	Box 118, Lund 221 00\\
	Sweden}
\email{Thomas.Benjamin.Turner@gmail.com}

\begin{abstract}
We study left-invariant foliations $\F$ on Riemannian Lie groups $G$ generated by a subgroup $K$. We are interested in such foliations which are conformal and with minimal leaves of codimension two.  We classify  such foliations $\F$ when the subgroup $K$ is one of the important $\SU2\times\SU2$, $\SU2\times\SLR2$, $\SU2\times\SO2$ or $\SLR2\times\SO2$.  By this we yield new multi-dimensional families of Lie groups $G$ carrying such foliations in each case.  These   foliations $\F$ produce local complex-valued harmonic morphisms on the corresponding Lie group $G$.
\end{abstract}

\subjclass[2010]{31B30, 53C43, 58E20}

\keywords{conformal foliations, harmonic morphisms}

\maketitle

\section{Introduction}\label{section-introduction}

We study Riemannian Lie groups $G$ equipped with a conformal foliation $\F$ generated by the left-translations of a subgroup $K$ of codimension two.  We are interested in such foliations with minimal leaves and hence inducing local complex-valued harmonic morphisms.
\smallskip 

Our principal aim is to investigate the validity of Conjecture \ref{conjecture-minimal-0} and show that it holds in the four important cases when $K$ is one of the Lie subgroups  $\SU2\times\SU 2$, $\SU 2\times\SLR 2$, $\SO 2\times\SU 2$ or $\SO 2\times\SLR 2$. In each case, we yield  multi-dimensional families of Riemannian Lie groups $G$ solving this interesting geometric problem i.e. carrying conformal foliations with minimal leaves of codimension two.

\begin{conjecture}\label{conjecture-minimal-0}
	Let $K$ be a subgroup of the Lie group $G$ generating a left-invariant conformal foliation $\F$ of $G$ of codimension two. If $K$ is semisimple then the foliation $\F$ is minimal.  If $K$ is semisimple and compact then $\F$ is totally geodesic.
\end{conjecture}

On this journey we also prove the following result.

\begin{theorem}\label{theorem-semisimple-0}
	Let $K$ be a semisimple Lie subgroup of $G$ of codimension two and $\F$ be the left-invariant foliation on $G$ generated by $K$.  Then $\F$ is Riemannian.
\end{theorem}

\section{Preliminaries}\label{section-preliminaries}

Let $M$ and $N$ be two manifolds of dimensions $m$ and $n$, respectively. A Riemannian metric $g$ on $M$ gives rise to the notion of a {\it Laplacian} on $(M,g)$ and real-valued {\it harmonic functions} $f:(M,g)\to\rn$. This can be generalised to the concept of {\it harmonic maps} $\phi:(M,g)\to (N,h)$ between Riemannian manifolds, which are solutions to a semi-linear system of partial differential equations, see \cite{Bai-Woo-book}.

\begin{definition}
A map $\phi:(M,g)\to (N,h)$ between Riemannian manifolds is called a {\it harmonic morphism} if, for any harmonic function $f:U\to\rn$ defined on an open subset $U$ of $N$ with $\phi^{-1}(U)$ non-empty, the composition $f\circ\phi:\phi^{-1}(U)\to\rn$ is harmonic.
\end{definition}

The following characterisation of harmonic morphisms between
Riemannian manifolds is due to B. Fuglede and T. Ishihara.  For the
definition of horizontal (weak) conformality we refer to
\cite{Bai-Woo-book}.

\begin{theorem}\cite{Fug,Ish}
	A map $\phi:(M,g)\to (N,h)$ between Riemannian manifolds is a
	harmonic morphism if and only if it is a horizontally (weakly)
	conformal harmonic map.
\end{theorem}

Let $(M,g)$ be a Riemannian manifold, $\V$ be an integrable distribution on $M$ and denote by $\H$ its orthogonal complement distribution on $M$.
As customary, we also use $\V$ and $\H$ to denote the orthogonal projections onto the corresponding subbundles of $TM$
and denote by $\F$ the foliation tangent to $\V$. Then the second fundamental form for $\V$ is given by
$$B^\V(E,F)=\frac 12\,\H(\nabla_EF+\nabla_FE)=\H(\nabla_EF)\qquad(E,F\in\V),$$
while the second fundamental form for $\H$ satisfies 
$$B^\H(X,Y)=\frac{1}{2}\,\V(\nabla_XY+\nabla_YX)\qquad(X,Y\in\H).$$
The foliation $\F$ tangent to $\V$ is said to be {\it conformal} if there is a
vector field $V\in \V$ such that $$B^\H=g\otimes V,$$ and
$\F$ is said to be {\it Riemannian} if $V=0$.
Furthermore, $\F$ is said to be {\it minimal} if $\text{trace}\ B^\V=0$ and
{\it totally geodesic} if $B^\V=0$. This is equivalent to the
leaves of $\F$ being minimal and totally geodesic submanifolds
of $M$, respectively.

It is well-known that the fibres of a horizontally conformal
map (resp.\ Riemannian submersion) give rise to a conformal foliation
(resp.\ Riemannian foliation). Conversely, the leaves of any
conformal foliation (resp.\ Riemannian foliation) are
locally the fibres of a horizontally conformal map
(resp.\ Riemannian submersion), see \cite{Bai-Woo-book}.

The next result of Baird and Eells gives the theory of
harmonic morphisms, with values in a surface,
a strong geometric flavour.

\begin{theorem}\cite{Bai-Eel}\label{theo:B-E}
	Let $\phi:(M^m,g)\to (N^2,h)$ be a horizontally conformal
	submersion from a Riemannian manifold to a surface. Then $\phi$ is
	harmonic if and only if $\phi$ has minimal fibres.
\end{theorem}

\section{Left-invariant foliations of codimension $2$}

Let $(G,g)$ be a Lie group equipped with a left-invariant Riemannian metric $g$ and $K$ be a subgroup of $G$.  Let $\k$ and $\g$ be the Lie algebras of $K$ and $G$, respectively.  Let $\m$ be the orthogonal complement of $\k$ in $\g$ with respect to the Riemannian metric $g$ on $G$.  By $\V$ we denote the integrable distribution generated by $\k$ and by $\H$ its  orthogonal distribution given by $\m$.  Further let $\F$ be the foliation of $G$ induced by $\V$.  For this situation we have the following result.

\begin{theorem}\label{theorem-semisimple}
Let $K$ be a semisimple Lie subgroup of $G$ of codimension two and $\F$ be the left-invariant foliation on $G$ generated by $K$.  Then $\F$ is Riemannian.
\end{theorem}

\begin{proof}
Since the subgroup $K$ is semisimple we know that its Lie algebra $\k$ satisfies $[\k,\k]=\k$.  It then follows from Remark 3.2 of
\cite{Gud-Sve-6} that the adjoint action of $\V=[\V,\V]$ has no $\H$-component.  The statement is an immediate consequence of this fact.
\end{proof}

Theorem \ref{theorem-semisimple} motivates the following conjecture. 

\begin{conjecture}\label{conjecture-minimal}
Let $K$ be a Lie subgroup of $G$ generating a left-invariant conformal foliation $\F$ of $G$ of codimension two. If $K$ is semisimple then the foliation $\F$ is minimal.  If $K$ is semisimple and compact then $\F$ is totally geodesic.
\end{conjecture}

We will investigate this conjecture in the four interesting  cases when $K$ is one of the Lie subgroups $\SU2\times\SU 2$, $\SU 2\times\SLR 2$, $\SU 2\times\SO 2$ or $\SLR 2\times\SO 2$ of $G$. 
The following table gives an explanation for our  different choices.
\medskip

\begin{center}
\begin{tabular}{c|c|c}
$K$ & \text{compact} & \text{non-compact} \\
\hline
\text{semisimple}  & $\SU 2\times\SU 2$ & $\SU 2\times\SLR 2$ \\
\text{non-semisimple}  & $\SU 2\times\SO 2$ & $\SLR 2\times\SO 2$ \\
\end{tabular}
\end{center}

\section{The case of $K=\SU2\times\SU 2$ in $G^8$.}\label{section-SU2-SU2}

Let $(G,g)$ be an eight-dimensional Riemannian Lie group and $K$ be its compact semisimple subgroup $\SU2\times\SU 2$ equipped with its standard Riemannian metric induced by the corresponding Killing form.  Let $\F$ be the left-invariant foliation on $G$ generated by the Lie subalgebra $\k=\su 2\times\su 2$ of $\g$.  Let $\{A,B,C,R,S,T,X,Y\}$ be an orthonormal basis for the Lie algebra $\g$ of $G$ such that the two copies of $\su 2$ are generated by $\{A,B,C\}$ and $\{R,S,T\}$, respectively. For these we have the following standard Lie bracket relations 
\begin{equation*}
\left[A,B\right]=2\,C,\ \ \left[C,A\right]=2\,B,\ \  \left[B,C\right]=2\,A,
\end{equation*}
\begin{equation*}
\left[R,S\right]=2\,T,\ \  \left[T,R\right]= 2\,S,\ \  \left[S,T\right]=2\,R.
\end{equation*}
For this choice of bases we have the following result.

\begin{proposition}\label{proposition-SU2*SU2}
Let G be an eight-dimensional Riemannian Lie group and $K$ be its six-dimensional subgroup  $\SU2\times\SU2$ equipped with its standard Riemannian metric induced by the Killing form. Then the Lie bracket relations for $\g$ take the following form
\begin{equation*}
\left[A,B\right]=2\,C,\ \ \left[C,A\right]=2\,B,\ \  \left[B,C\right]=2\,A,
\end{equation*}
\begin{equation*}
\left[R,S\right]=2\,T,\ \  \left[T,R\right]= 2\,S,\ \  \left[S,T\right]=2\,R,
\end{equation*}
\begin{equation*}
\left[A,X\right] =-b_{11}B -c_{11}C,\ \
\left[A,Y\right] = -b_{21}B  -c_{21}C,
\end{equation*}
\begin{equation*}
\left[B,X\right] = b_{11}A - c_{12}C,\ \
\left[B,Y\right] = b_{21}A - c_{22}C,
\end{equation*}
\begin{equation*}
\left[C,X\right] = c_{11}A + c_{12}B,\ \
\left[C,Y\right] = c_{21}A + c_{22}B,
\end{equation*}
\begin{equation*}
\left[R,X\right] = -s_{14}S - t_{14}T,\ \
\left[R,Y\right] = -s_{24}S - t_{24}T,
\end{equation*}
\begin{equation*}
\left[S,X\right] = s_{14}R - t_{15}T,\ \
\left[S,Y\right] = s_{24}R - t_{25}T,
\end{equation*}
\begin{equation*}
\left[T,X\right] = t_{14}R + t_{15}S,\ \
\left[T,Y\right] = t_{24}R + t_{25}S,
\end{equation*}
\begin{equation*}
\left[X,Y\right] = {\rho}X +\theta_{1}A + \theta_{2}B + \theta_{3}C + \theta_{4}R + \theta_{5}S + \theta_{6}T.
\end{equation*}
Here the 13 different real coefficients $b_{jk},c_{jk},s_{jk},t_{jk},\rho$ are arbitrary and $\theta_1,\dots,\theta_6$ are given by
\begin{equation*}
\begin{pmatrix}
\theta_{1}\\
\theta_{2}\\
\theta_{3}\\
\theta_{4}\\
\theta_{5}\\
\theta_{6}
\end{pmatrix}
=
\frac{1}{2}\begin{pmatrix}
-{\rho}c_{12} + b_{11}c_{21} - b_{21}c_{11}\\
{\rho}c_{11} + b_{11}c_{22} - b_{21}c_{12}\\
-{\rho}b_{11} + c_{11}c_{22} - c_{12}c_{21}\\
-{\rho}t_{15} + s_{14}t_{24} - s_{24}t_{14}\\
{\rho}t_{14} + s_{14}t_{25} - s_{24}t_{15}\\
-{\rho}s_{14} + t_{14}t_{25} - t_{15}t_{24}
\end{pmatrix}.
\end{equation*}
\end{proposition}

\begin{proof}
Since $K=\SU 2\times\SU 2$ is semisimple we know from the proof of Theorem \ref{theorem-semisimple} that the adjoint action of $\V=[\V,\V]$ has no
$\H$-component. This tells us that the corresponding  Lie bracket relations take the form 
\begin{equation*}
	\left[A,B\right]=2\,C,\ \ \left[C,A\right]=2\,B,\ \  \left[B,C\right]=2\,A,
\end{equation*}
\begin{equation*}
	\left[R,S\right]=2\,T,\ \  \left[T,R\right]= 2\,S,\ \  \left[S,T\right]=2\,R,
\end{equation*}
\begin{eqnarray*}
\left[A,X\right] &=& a_{11}A + a_{12}B + a_{13}C + a_{14}R + a_{15}S + a_{16}T,\\
\left[A,Y\right] &=& a_{21}A + a_{22}B + a_{23}C + a_{24}R + a_{25}S + a_{26}T,\\
\left[B,X\right] &=& b_{11}A + b_{12}B + b_{13}C + b_{14}R + b_{15}S + b_{16}T,\\
\left[B,Y\right] &=& b_{21}A + b_{22}B + b_{23}C + b_{24}R + b_{25}S + b_{26}T,\\
\left[C,X\right] &=& c_{11}A + c_{12}B + c_{13}C + c_{14}R + c_{15}S + c_{16}T,\\
\left[C,Y\right] &=& c_{21}A + c_{22}B + c_{23}C + c_{24}R + c_{25}S + c_{26}T,\\
\left[R,X\right] &=& r_{11}A + r_{12}B + r_{13}C + r_{14}R + r_{15}S + r_{16}T,\\
\left[R,Y\right] &=& r_{21}A + r_{22}B + r_{23}C + r_{24}R + r_{25}S + r_{26}T,\\
\left[S,X\right] &=& s_{11}A + s_{12}B + s_{13}C + s_{14}R + s_{15}S + s_{16}T,\\
\left[S,Y\right] &=& s_{21}A + s_{22}B + s_{23}C + s_{24}R + s_{25}S + s_{26}T,\\
\left[T,X\right] &=& t_{11}A + t_{12}B + t_{13}C + t_{14}R + t_{15}S + t_{16}T,\\
\left[T,Y\right] &=& t_{21}A + t_{22}B + t_{23}C + t_{24}R + t_{25}S + t_{26}T,\\
\left[X,Y\right] &=& {\rho}X+\theta_{1}A+\theta_{2}B+\theta_{3}C+\theta_{4}R + \theta_{5}S + \theta_{6}T.
\end{eqnarray*}

The Jacobi identites involving the vector fields $A,B,C,X\in\g$ provide us with the following interesting identities
\begin{eqnarray*}\label{jiabx}
	0&=&\jid(ABX)\\&=&2\,((c\ss11 + a\ss13)\,A + (c\ss12 + b\ss13)\,B + (c\ss13 - a\ss11 - b\ss12)\,C\\
	& &\quad +c\ss14R + c\ss15S + c\ss16T),
\end{eqnarray*}
\begin{eqnarray*}\label{jicax}
	0&=&\jid(CAX)\\ &=& 2\,((b\ss11+a\ss12)A +(b\ss12 - a\ss11 - c\ss13)B + (b\ss13 + c\ss12)C\\
	& & \quad + b\ss14R + b\ss15S + b\ss16T),
\end{eqnarray*}
\begin{eqnarray*}\label{jibcx}
	0&=&\jid(BCX)\\
	&=&2\,((a\ss11 - b\ss12 - c\ss13)A + (a\ss12 + b\ss11)B + (a\ss13 + c\ss11)C\\
	& &\quad + a\ss14R + a\ss15S + a\ss16T).
\end{eqnarray*}
From these we can immediately see that 
\begin{equation*}
\begin{pmatrix}
a\ss14\\
a\ss15\\
a\ss16\\
\end{pmatrix}
=
\begin{pmatrix}
b\ss14\\
b\ss15\\
b\ss16\\
\end{pmatrix}
=
\begin{pmatrix}
c\ss14\\
c\ss15\\
c\ss16\\
\end{pmatrix}= 0.
\end{equation*}

Then by the symmetry in $X$ and $Y$ we see that this holds for $a_{i}, b_{i}, c_{i}$ where $i \in \{24, 25, 26\}$. Additionally, we see by the symmetry in $R, S, T$ this also holds for $r_{j}, s_{j}, t_{j}$ for $j \in \{11,12, 13, 21, 22, 23\}$.
We also see that these Jacobi identities are satisfied according to the following system of equations
\begin{equation*}
\begin{pmatrix}
a\ss11 - b\ss12 - c\ss13\\
a\ss11 - b\ss12 + c\ss13\\
a\ss11 + b\ss12 - c\ss13\\
a\ss12 + b\ss11\\
a\ss13 + c\ss11\\
b\ss13 + c\ss12
\end{pmatrix} = 0.
\end{equation*}
Solving this system we obtain
\begin{equation}
\begin{pmatrix}
a\ss11\\
b\ss12\\
c\ss13
\end{pmatrix}= 0,
\hskip1cm
\begin{pmatrix}
a\ss12\\
a\ss13\\
b\ss13
\end{pmatrix} = 
\begin{pmatrix}
-b\ss11\\
-c\ss11\\
-c\ss12
\end{pmatrix}.
\end{equation}
Due to the symmetry in both $X$ and $Y$ we see that the Jacobi identities involving $A, B, C, Y$ are equivalent to the following system of equations
\begin{equation*}
\begin{pmatrix}
a\ss21 - b\ss22 - c\ss23\\
a\ss21 - b\ss22 + c\ss23\\
a\ss21 + b\ss22 - c\ss23\\
a\ss22 + b\ss21\\
a\ss23 + c\ss21\\
b\ss23 + c\ss22
\end{pmatrix} = 0.
\end{equation*}
Solving this we obtain the following results from satisfying the Jacobi identies
\begin{equation}
\begin{pmatrix}
a\ss21\\
b\ss22\\
c\ss23
\end{pmatrix}
= 0, \hskip1cm
\begin{pmatrix}
a\ss22\\
a\ss23\\
b\ss23
\end{pmatrix} = 
\begin{pmatrix}
-b\ss21\\
-c\ss21\\
-c\ss22
\end{pmatrix}.
\end{equation}
By the symmetry in $A, B, C$ and $R, S, T$ we have that the Jacobi identities involving $R, S, T, X$ are equivalent to the following system of equations
\begin{equation*}
\begin{pmatrix}
r\ss14 - s\ss15 - t\ss16\\
r\ss14 - s\ss15 + t\ss16\\
r\ss14 + s\ss15 - t\ss16\\
r\ss15 + s\ss14\\
r\ss16 + t\ss14\\
s\ss16 + t\ss15
\end{pmatrix} = 0.
\end{equation*}
Solving this we then yield the following results from satisfying the Jacobi identies
\begin{equation}
\begin{pmatrix}
r\ss14\\
s\ss15\\
t\ss16
\end{pmatrix}= 0,
\hskip1cm
\begin{pmatrix}
r\ss15\\
r\ss16\\
s\ss16
\end{pmatrix} = 
\begin{pmatrix}
-s\ss14\\
-t\ss14\\
-t\ss15
\end{pmatrix}.
\end{equation}
Again by the symmetry in $X$ and $Y$ we see that the Jacobi identities involving $R, S, T, Y$ are equivalent to the following system of equations
\begin{equation*}
\begin{pmatrix}
r\ss24 - s\ss25 - t\ss26\\
r\ss24 - s\ss25 + t\ss26\\
r\ss24 + s\ss25 - t\ss26\\
r\ss25 + s\ss24\\
r\ss26 + t\ss24\\
s\ss26 + t\ss25
\end{pmatrix} = 0.
\end{equation*}
Solving this we then obtain the following results from satisfying the Jacobi identies
\begin{equation}\label{RSTres}
\begin{pmatrix}
r\ss24\\
s\ss25\\
t\ss26
\end{pmatrix}=0,
\hskip1cm
\begin{pmatrix}
r\ss25\\
r\ss26\\
s\ss26
\end{pmatrix} = 
\begin{pmatrix}
-s\ss24\\
-t\ss24\\
-t\ss25
\end{pmatrix}.
\end{equation}
Now considering the Jacobi identities involving both $X$ and $Y$ we see that
\begin{eqnarray*}
	0&=&\jid(AXY)\\
	&=&(\,{\rho}b\ss11 - c\ss11c\ss22 + c\ss21c\ss12 + 2\theta\ss3\, )\,B + ({\rho}c\ss11 + b\ss11c\ss22 - b\ss21c\ss12 - 2\theta\ss2\,)C,
\end{eqnarray*}
\begin{eqnarray*}
	0&=&\jid(BXY)\\
	&=&-(\,{\rho}b\ss11 - c\ss11c\ss22 + c\ss21c\ss12 + 2\theta\ss3\, )\,A + ({\rho}c\ss12 - b\ss11c\ss21 + b\ss21c\ss11 + 2\theta\ss1\,)C.
\end{eqnarray*}
Satisfying these Jacobi identities is equivalent to the following system of equations
\begin{equation*}
\begin{pmatrix}
{\rho}c\ss12 - b\ss11c\ss21 + b\ss21c\ss11 + 2\theta\ss1\,\\
{\rho}c\ss11 + b\ss11c\ss22 - b\ss21c\ss12 - 2\theta\ss2\,\\
{\rho}b\ss11 - c\ss11c\ss22 + c\ss21c\ss12 + 2\theta\ss3\,
\end{pmatrix} = 0,
\end{equation*}
from which we get
\begin{equation*}
\begin{pmatrix}
\theta\ss1\,\\
\theta\ss2\,\\
\theta\ss3\,
\end{pmatrix} = 
\frac{1}{2}\begin{pmatrix}
-{\rho} c\ss12 + b\ss11 c\ss21 - b\ss21 c\ss11\\
{\rho} c\ss11 + b\ss11 c\ss22 - b\ss21 c\ss12\\
-{\rho} b\ss11 + c\ss11 c\ss22 - c\ss12 c\ss21
\end{pmatrix}.
\end{equation*}
Then by the symmetry of the triples $(A, B,C )$ and $(R,S,T)$ we find that 
\begin{equation*}
\begin{pmatrix}
\theta\ss4\,\\
\theta\ss5\,\\
\theta\ss6\,
\end{pmatrix} = 
\frac{1}{2}\begin{pmatrix}
-{\rho}t_{15} + s_{14}t_{24} - s_{24}t_{14}\\
{\rho}t_{14} + s_{14}t_{25} - s_{24}t_{15}\\
-{\rho}s_{14} + t_{14}t_{25} - t_{15}t_{24}
\end{pmatrix}.
\end{equation*}
These calculations provide us with the stated result.
\end{proof}

The following Theorem \ref{theorem-SU2*SU2} supports the case of Conjecture \ref{conjecture-minimal} that when the semisimple subgroup $K$ is compact, the resulting foliation is totally geodesic. Together with Proposition \ref{proposition-SU2*SU2}, this provides us with a new 13-dimensional family of $8$-dimensional Lie groups carrying a conformal foliation with minimal leaves of codimension two. This is generated by the 13 real parameters $b_{jk}, c_{jk}, s_{jk}, t_{jk}, \rho$ given in the statement of Proposition \ref{proposition-SU2*SU2}.

\begin{theorem}\label{theorem-SU2*SU2}
Let $G$ be an eight-dimensional Riemannian Lie group containing the subgroup $\SU2\times\SU2\,.$ And let $\V$ be the left-invariant distribution generated by the Lie algebra $\su2\times\su2\,.$ Then the resulting foliation $\F$ tangent to $\V$ is Riemannian and totally geodesic.
\end{theorem}

\begin{proof}
It immediately follows from Theorem \ref{theorem-semisimple} that the foliation $\F$ is Riemannian, thus it suffices to show that $\F$ is totally geodesic. For smooth vector fields $ A,B \in \V$ in the vertical distribution we have 
\begin{eqnarray*}
\BV(AB)&=&\frac 12\,\H (\nab AB+\nab BA)\\
&=& \frac{1}{2}\Big((\scalar{\lieb(XA)}{B} + \scalar{\lieb(XB)}{A})X\\
& & \qquad + (\scalar{\lieb(YA)}{B} + \scalar{\lieb(YB)}{A})Y\Big).
\end{eqnarray*}

Evaluating this for the basis elements in $\V$ and recalling the simplifications made in Proposition \ref{proposition-SU2*SU2} we determine
\begin{eqnarray*}
\BV(AA) \, &=&-(a_{11}X+a_{21}Y) = 0,\\
\BV(AB) \, &=&\, -\frac{1}{2}((a\ss12+b\ss11)X + (a\ss22 + b\ss21)Y) = 0,\\
\BV(AC) \, &=&\, -\frac{1}{2}((a\ss13+c\ss11)X + (a\ss23 + c\ss21)Y) = 0,\\
\BV(AR) \, &=&\, -\frac{1}{2}((a\ss14+r\ss11)X + (a\ss24 + r\ss21)Y) = 0,\\
\BV(AS) \, &=&\, -\frac{1}{2}((a\ss15+s\ss11)X + (a\ss25 + s\ss21)Y) = 0,\\
\BV(AT) \, &=&\, -\frac{1}{2}((a\ss16+t\ss11)X + (a\ss26 + t\ss21)Y) = 0,\\
\BV(BB) \, &=&\, -(b\ss12X + b\ss22Y) = 0,\\
\BV(BC) \, &=&\, -\frac{1}{2}((b\ss13+c\ss12)X + (b\ss23 + c\ss22)Y) = 0,\\
\BV(BR) \, &=&\, -\frac{1}{2}((b\ss14+r\ss12)X + (b\ss24 + r\ss22)Y) = 0,\\
\BV(BS) \, &=&\, -\frac{1}{2}((b\ss15+s\ss12)X + (b\ss25 + s\ss22)Y) = 0,\\
\BV(BT) \, &=&\, -\frac{1}{2}((b\ss16+t\ss12)X + (b\ss26 + t\ss22)Y) = 0,\\
\BV(CC) \, &=&\, -(c\ss13X + c\ss23Y) = 0,\\
\BV(CR) \, &=&\, -\frac{1}{2}((c\ss14+r\ss13)X + (c\ss24 + r\ss23)Y) = 0,\\
\BV(CS) \, &=&\, -\frac{1}{2}((c\ss15+s\ss13)X + (c\ss25 + s\ss23)Y) = 0,\\
\BV(CT) \, &=&\, -\frac{1}{2}((c\ss16+t\ss13)X + (c\ss26 + t\ss23)Y) = 0,\\
\BV(RR) \, &=&\, -(r\ss14X + r\ss24Y) = 0,\\
\BV(RS) \, &=&\, -\frac{1}{2}((r\ss15+s\ss14)X + (r\ss25 + s\ss24)Y) = 0,\\
\BV(RT) \, &=&\, -\frac{1}{2}((r\ss16+t\ss14)X + (r\ss26 + t\ss24)Y) = 0,\\
\BV(SS) \, &=&\, -(s\ss15X + s\ss25Y) = 0,\\
\BV(ST) \, &=&\, -\frac{1}{2}((s\ss16+t\ss15)X + (s\ss26 + t\ss25)Y) = 0,\\
\BV(TT) \, &=&\, -(t\ss16X + t\ss26Y) = 0.
\end{eqnarray*}
From this we see that the second fundamental form $B^\V$ of the vertical distribution vanishes. Thus the foliation $\F$ is totally geodesic.	
\end{proof}

\section{The case of $K=\SU2\times\SLR 2$ in $G^8$.}\label{section-SU2-SL2}

Let $(G,g)$ be an $8$-dimensional Riemannian Lie group and $K$ be its $6$-dimensional non-compact semisimple subgroup $\SU2\times\SLR 2$ equipped with its standard Riemannian metric.  Further let $\F$ be the left-invariant foliation on $G$ generated by the Lie subalgebra $\k=\su 2\times\slr 2$ of $\g$.  Let $\{A,B,C,R,S,T,X,Y\}$ be an orthonormal basis for the Lie algebra $\g$ of $G$ such that $\su 2$ is generated by $\{A,B,C\}$ and $\slr 2$ by $\{R,S,T\}$, respectively, with their standard Lie bracket relations 
\begin{equation*}
\left[A,B\right]=2\,C,\ \ \left[C,A\right]=2\,B,\ \  \left[B,C\right]=2\,A,
\end{equation*}
\begin{equation*}
\left[R,S\right]=2\,T,\ \  \left[T,R\right]= 2\,S,\ \  \left[S,T\right]=-2\,R.
\end{equation*}

This case might look much the same as that of Section \ref{section-SU2-SU2}, but the minus in the last equation changes everything.  Both the subgroups are semisimple, but $\SU 2\times\SU 2$ is compact and  $\SU 2\times\SLR 2$ is not.  For the latter case we have the next result.

\begin{proposition}\label{proposition-SU2*SLR2}
Let G be an eight-dimensional Riemannian Lie group and $K$ be its six-dimensional Lie subgroup  $\SU2\times\SLR 2$ equipped with its standard Riemannian metric. Then the Lie bracket relations for $\g$ take the following form
\begin{equation*}
\left[A,B\right]=2\,C,\ \ \left[C,A\right]=2\,B,\ \  \left[B,C\right]=2\,A,
\end{equation*}
\begin{equation*}
\left[R,S\right]=2\,T,\ \  \left[T,R\right]= 2\,S,\ \  \left[S,T\right]=-2\,R,
\end{equation*}
\begin{equation*}
\left[A,X\right] =-b_{11}B -c_{11}C,\ \
\left[A,Y\right] = -b_{21}B  -c_{21}C,
\end{equation*}
	\begin{equation*}
	\left[B,X\right] = b_{11}A - c_{12}C,\ \
	\left[B,Y\right] = b_{21}A - c_{22}C,
	\end{equation*}
	\begin{equation*}
	\left[C,X\right] = c_{11}A + c_{12}B,\ \
	\left[C,Y\right] = c_{21}A + c_{22}B,
	\end{equation*}
	\begin{equation*}
	\left[R,X\right] = -s_{14}S - t_{14}T,\ \
	\left[R,Y\right] = -s_{24}S - t_{24}T,
	\end{equation*}
	\begin{equation*}
	\left[S,X\right] = s_{14}R - t_{15}T,\ \
	\left[S,Y\right] = s_{24}R - t_{25}T,
	\end{equation*}
	\begin{equation*}
	\left[T,X\right] = t_{14}R + t_{15}S,\ \
	\left[T,Y\right] = t_{24}R + t_{25}S,
	\end{equation*}
	\begin{equation*}
	\left[X,Y\right] = {\rho}X +\theta_{1}A + \theta_{2}B + \theta_{3}C + \theta_{4}R + \theta_{5}S + \theta_{6}T.
	\end{equation*}
Here the 13 different real coefficients $b_{jk},c_{jk},s_{jk},t_{jk},\rho$ are arbitrary and $\theta_1,\dots,\theta_6$ are given by
\begin{equation*}
\begin{pmatrix}
\theta_{1}\\
\theta_{2}\\
\theta_{3}\\
\theta_{4}\\
\theta_{5}\\
\theta_{6}
\end{pmatrix}
=
\frac{1}{2}\begin{pmatrix}
-{\rho}c_{12} + b_{11}c_{21} - b_{21}c_{11}\\
{\rho}c_{11} + b_{11}c_{22} - b_{21}c_{12}\\
-{\rho}b_{11} + c_{11}c_{22} - c_{12}c_{21}\\
-{\rho}t_{15} - s_{14}t_{24} + s_{24}t_{14}\\
-{\rho}t_{14} - s_{14}t_{25} + s_{24}t_{15}\\
{\rho}s_{14} - t_{14}t_{25} + t_{15}t_{24}\\
\end{pmatrix}.
\end{equation*}	
\end{proposition}

\begin{proof}
The arguments needed here are exactly the same as already provided in the proof of Proposition \ref{proposition-SU2*SU2}, for the details see \cite{Tur}.
\end{proof}

The following Theorem \ref{theorem-SU2*SLR2} supports the case of Conjecture \ref{conjecture-minimal}, that when the subgroup $K$ is semisimple then the resulting foliation is minimal. Together with Proposition \ref{proposition-SU2*SLR2}, this gives a new 13-dimensional family of $8$-dimensional Lie groups carrying a conformal foliation with minimal leaves of codimension two.  This is generated by the 13 real parameters $b_{jk}, c_{jk}, s_{jk}, t_{jk}, \rho$ given in the statement of Proposition \ref{proposition-SU2*SLR2}.

\begin{theorem}\label{theorem-SU2*SLR2}
Let $G$ be an eight-dimensional Riemannian Lie group containing the subgroup $\SU2\times\SLR2.$ And let $\V$ be the left-invariant distribution generated by the Lie algebra $\su2\times\slr2$. Then the resulting foliation $\F$ tangent to $\V$ is Riemannian and  minimal. Furthermore, $\F$ is totally geodesic if and only if $s_{14}=s_{24}=t_{14}=t_{24}=0$, where $s_{14}, s_{24}, t_{14}, t_{24}$ are the parameters defined in Proposition \ref{proposition-SU2*SLR2}.
\end{theorem}

\begin{proof}
It immediately follows from Theorem \ref{theorem-semisimple} that $\F$ is Riemannian.   Now employing the Koszul formula, for the Levi-Civita connection, we obtain  
\begin{eqnarray*}
& & \trace\, B^\V \\
&=& \H(\nab AA + \cdots + \nab TT)\\
&=& (\scalar{\nab AA}{X} + \cdots + \scalar{\nab TT}{X})X + (\scalar{\nab AA}{Y} + \cdots + \scalar{\nab TT}{Y})Y\\
&=& (\scalar{\lieb(XA)}{A} + \cdots + \scalar{\lieb(XT)}{T})X + (\scalar{\lieb(YA)}{A} + \cdots + \scalar{\lieb(YT)}{T})Y\\
&=& -(a\ss11 + b\ss12 + c\ss13 + r\ss14 + s\ss15 + t\ss16)X\\
& &\qquad - (a\ss21 + b\ss22 + c\ss23 + r\ss24 + s\ss25 + t\ss26)Y.
\end{eqnarray*}
Therefore, by the simiplifications made to the Lie bracket relations in Proposition \ref{proposition-SU2*SLR2}, we see that $\F$ is clearly minimal i.e. $\trace \,B^\V = 0$.

We now check when the foliation $\F$ is totally geodesic by using the same method as in the proof of Theorem \ref{theorem-SU2*SU2}.  We find that all but the two following evaluations to be zero.
\begin{eqnarray*}
\BV(RS) &=& -s\ss14 X - s\ss24 Y,\\
\BV(RT) &=& -t\ss14 X - t\ss24 Y.
\end{eqnarray*}
This proves the statement.
\end{proof}

\section{The case of $K=\SU 2\times\SO 2$ in $G^6$.}\label{section-SU2*SO2}

Let $(G,g)$ be a six-dimensional Riemannian Lie group and $K$ be its compact subgroup $\SU 2\times\SO 2$ equipped with its standard Riemannian metric. Further let $\F$ be the left-invariant foliation on $G$ generated by the Lie subalgebra $\k=\su 2\times\so 2$ of $\g$.  Let $\{A,B,C,T,X,Y\}$ be an orthonormal basis for the Lie algebra $\g$ of $G$ such that $\su 2$ is generated by $\{A,B,C\}$ and $\so 2$ by $\{T\}$, respectively. For the Lie algebra $\k$ of $K$ we have the following standard non-vanishing Lie bracket relations 
\begin{equation*}
	\left[A,B\right]=2\,C,\ \ \left[C,A\right]=2\,B,\ \  \left[B,C\right]=2\,A.
\end{equation*}

In this case the subalgebra $\k$ is not semisimple since $\so2$ is abelian, therefore we find that Remark 3.2 of \cite{Gud-Sve-6} only applies to $\su2$ i.e. the semisimple component of $\k$. Due to this the remaining Lie bracket relations, given by $\mathrm{ad}_\V(\H)$ and $\mathrm{ad}_\H(\H)$, are of the following form
\begin{eqnarray*}
	\left[A,X\right] &=& a_{11}A + a_{12}B + a_{13}C + a_{14}T,\\
	\left[A,Y\right] &=& a_{21}A + a_{22}B + a_{23}C + a_{24}T,\\
	\left[B,X\right] &=& b_{11}A + b_{12}B + b_{13}C + b_{14}T,\\
	\left[B,Y\right] &=& b_{21}A + b_{22}B + b_{23}C + b_{24}T,\\
	\left[C,X\right] &=& c_{11}A + c_{12}B + c_{13}C + c_{14}T,\\
	\left[C,Y\right] &=& c_{21}A + c_{22}B + c_{23}C + c_{24}T,\\
	\left[T,X\right] &=&  x_{1}X + y_{1}Y+t_{11}A+t_{12}B+t_{13}C+t_{14}T,\\
	\left[T,Y\right] &=&  x_{2}X + y_{2}Y+t_{21}A+t_{22}B+t_{23}C+t_{24}T,\\
	\left[X,Y\right] &=& {\rho}X+\theta_{1}A+\theta_{2}B+\theta_{3}C+\theta_{4}T.
\end{eqnarray*}
By a method similar to the one used in the case when $K$ is semisimple, we find that the Lie bracket relations involving $X$ and $Y$ can be greatly simplified.

\begin{proposition}\label{proposition-SU2*SO2}
Let $G$ be a six-dimensional Riemannian Lie group and $K$ its four-dimensional subgroup $\SU2\times\SO2$. Then the Lie bracket relations for the Lie algebra $\g$ of $G$ can be written as
$$\left[A,B\right]=2\,C,\ \ \left[C,A\right]=2\,B,\ \  \left[B,C\right]=2\,A,$$
$$\left[A,X\right] = -b_{11}B -c_{11}C,\ \ 
\left[A,Y\right] = -b_{21}B -c_{21}C,$$
$$\left[B,X\right] = b_{11}A -c_{12}C,\ \ 
\left[B,Y\right] = b_{21}A -c_{22}C,$$
$$\left[C,X\right] = c_{11}A + c_{12}B,\ \ 
\left[C,Y\right] = c_{21}A + c_{22}B,$$
\begin{eqnarray*}
\left[T,X\right] &=&  x_{1}\,X + y_{1}\,Y +\frac{1}{2}((-c\ss12x_1 - c\ss22y_1)A\\
& & +(c\ss11x_1 + c\ss21y_1)B+(-b\ss11x_1 - b\ss21y_1)C),\\
\left[T,Y\right] &=& x_2\,X - x_1\,Y + \frac{1}{2}((-c\ss12x_2 + c\ss22x_1)A\\
& & + (c\ss11x_2 -c\ss21x_1)B+(-b\ss11x_2 + b\ss21x_1)C) -\rho\,T,\\
\lieb(XY) &=& {\rho}X+\theta_{1}A+\theta_{2}B+\theta_{3}C+\theta_{4}T.
\end{eqnarray*}
Here the 11 different real coefficients $b_{jk},c_{jk},x_1,x_2,y_1,\rho,\theta_4$ are arbitrary and $\theta_1,\theta_2,\theta_3$ are given by
\begin{equation*}
\begin{pmatrix}
\theta_1\\
\theta_2\\
\theta_3
\end{pmatrix}
=\frac{1}{2}
\begin{pmatrix}
-{\rho}c\ss12 + b\ss11c\ss21 - b\ss21c\ss11\\
{\rho}c_{11} + b_{11}c_{22} - b_{21}c_{12}\\
-{\rho}b_{11} + c_{11}c_{22} - c_{12}c_{21}
\end{pmatrix}.
\end{equation*}
\end{proposition}

\begin{proof}
The arguments needed here are exactly the same as already provided in the proof of Proposition \ref{proposition-SU2*SU2}, for the details see \cite{Tur}.
\end{proof}

\begin{theorem}\label{theorem-SU2*SO2}
Let $G$ be a six-dimensional Riemannian Lie group containing the subgroup $\SU2\times\SO2\,.$ And let $\V$ be the left-invariant distribution generated by the Lie algebra $\su2\times\so2$. The resulting foliation $\F$ tangent to $\V$ is
\begin{enumerate}
\item conformal if and only if $y_1+x_2=0$ and $x_1=0$,
\item Riemannian if and only if it is conformal,
\item minimal if and only if $\rho=0$, and
\item totally geodesic if and only if $$\rho = y_1b_{11}=y_1b_{21}= y_1c_{11}=y_1c_{12}=y_1c_{21}=y_1c_{22} = 0\,$$
\end{enumerate}
where the parameters $b_{11}\,, b_{21}\,, c_{11}, c_{12}\,, c_{21}\,, c_{22}\,, x_1\,, x_2\,, y_1\,,$ and $\rho$ are defined in Proposition \ref{proposition-SU2*SO2}.
\end{theorem}

\begin{proof}
We begin by noting that since $K$ is not semisimple the condition of conformality is not immediately given. Evaluating the second fundamental form $B^\H$ of the horizontal distribution we find that
$$\BH(XX) = x_1\,T,\ \ \BH(YY) = -x_1\,T,$$
$$\BH(XY) = \frac{1}{2}(y_1+x_2)\,T.$$
From this we determine that $\F$ is conformal if and only if $y_1+x_2=0$ and $x_1 = 0$.  This is also equivalent to $\F$ being Riemannian.

Next we check whether the foliation $\F$ is totally geodesic or not, by evaluating the second fundamental form of the vertical distribution $\V$.  Here we get 
$$\BV(AA) = 0,\ \ \BV(AB) = 0,\ \ \BV(AC) = 0,$$
$$\BV(AT) = \frac{y_1}{4}(c\ss22\,X - c\ss12\,Y),\ \ 
\BV(BB) = 0,\ \ \BV(BC) = 0,$$
$$\BV(BT) = \frac{y_1}{4}(-c\ss21\,X + c\ss11\,Y),\ \ 
\BV(CC) = 0,$$
$$\BV(CT)= \frac{y_1}{4}(b\ss21\,X - b\ss11\,Y),\ \ 
\BV(TT)= \rho\,Y.$$
From these results it is immediately clear that $\F$ is minimal if and only if $\rho = 0$. In addition, we see that the foliation is totally geodesic if and only if 
$$\rho = y_1b_{11}=y_1b_{21}= y_1c_{11}=y_1c_{12} =y_1c_{21}=y_1c_{22} = 0.$$
\end{proof}

Proposition \ref{proposition-SU2*SO2} and Theorem \ref{theorem-SU2*SO2} provide a new $8$-dimensional family of $6$-dimensional Lie groups carrying a conformal foliation with minimal leaves of codimension two.  This is generated by the  $8$ parameters $b_{11}$, $b_{21}$, $c_{11}$, $c_{21}$, $c_{12}$, $c_{22}$, $y_1$, $\theta_4$ which are defined in the statement of Proposition \ref{proposition-SU2*SO2}.

\section{The case of $K=\SLR 2\times\SO 2$ in $G^6$.}\label{section-SLR2*SO2}

Let $(G,g)$ be a six-dimensional Riemannian Lie group and $K$ be its non-compact subgroup $\SLR 2\times\SO 2$ equipped with its standard Riemannian metric. Further let $\F$ be the left-invariant foliation on $G$ generated by the Lie subalgebra $\k=\slr 2\times\so 2$ of $\g$.  Let $\{A,B,C,T,X,Y\}$ be an orthonormal basis for the Lie algebra $\g$ of $G$ such that $\slr 2$ is generated by $\{A,B,C\}$ and $\so 2$ by $\{T\}$, respectively. For the Lie algebra $\k$ of $K$ we have the following standard non-vanishing Lie bracket relations 
\begin{equation*}
\left[A,B\right]=2\,C,\ \ \left[C,A\right]=2\,B,\ \  \left[B,C\right]=-2\,A.
\end{equation*}

The subgroup $K=\SLR 2\times\SO 2$ is not semisimple.  Due to this the remaining bracket relations, given by $\mathrm{ad}_\V(\H)$ and $\mathrm{ad}_\H(\H)$, are of the form
\begin{eqnarray*}
	\left[A,X\right] &=& a_{11}A + a_{12}B + a_{13}C + a_{14}T,\\
	\left[A,Y\right] &=& a_{21}A + a_{22}B + a_{23}C + a_{24}T,\\
	\left[B,X\right] &=& b_{11}A + b_{12}B + b_{13}C + b_{14}T,\\
	\left[B,Y\right] &=& b_{21}A + b_{22}B + b_{23}C + b_{24}T,\\
	\left[C,X\right] &=& c_{11}A + c_{12}B + c_{13}C + c_{14}T,\\
	\left[C,Y\right] &=& c_{21}A + c_{22}B + c_{23}C + c_{24}T,\\
	\left[T,X\right] &=&  x_{1}X + y_{1}Y+t_{11}A+t_{12}B+t_{13}C+t_{14}T,\\
	\left[T,Y\right] &=&  x_{2}X + y_{2}Y+t_{21}A+t_{22}B+t_{23}C+t_{24}T,\\
	\left[X,Y\right] &=& {\rho}X+\theta_{1}A+\theta_{2}B+\theta_{3}C+\theta_{4}T.
\end{eqnarray*}

Using exactly the same method as the one for  Proposition \ref{proposition-SU2*SO2}, the remaining Lie bracket relations involving $X$ and $Y$ can be simplified as follows.

\begin{proposition}\label{proposition-SLR2*SO2}
Let G be a six-dimensional Riemanian Lie group and $K$ be the four-dimensional subgroup $\SLR2\times\SO2$ of $G$. Then the Lie brackets for $\g$ take the form
\begin{equation*}
\left[A,B\right]=2\,C,\ \ \left[C,A\right]=2\,B,\ \  \left[B,C\right]=-2\,A.
\end{equation*}
$$\left[A,X\right] = b_{11}B + c_{11}C,\ \ 
\left[A,Y\right] = b_{21}B + c_{21}C,$$
$$\left[B,X\right] = b_{11}A -c_{12}C,\ \ 
\left[B,Y\right] = b_{21}A -c_{22}C,$$
$$\left[C,X\right] = c_{11}A + c_{12}B,\ \ 
\left[C,Y\right] = c_{21}A + c_{22}B,$$
\begin{eqnarray*}
\left[T,X\right] &=&  x_1\,X + y_1\,Y +\frac{1}{2}((-c\ss12x_1 -c\ss22y_1)A\\
& &+(-c\ss11x_1 -c\ss21y_1)B+(b\ss11x_1 + b\ss21y_1)C),\\
\left[T,Y\right] &=& x_2\,X - x_1\,Y +\frac{1}{2}((c\ss22x_1 - c\ss12x_2)A\\
	& &+(c\ss21x_1 - c\ss11x_2)B+(-b\ss21x_1 + b\ss11x_2)C) - \rho\,T,\\
	\lieb(XY) &=& {\rho}X + \theta_1 A + \theta_2 B + \theta_3 C + \theta_4 T,
\end{eqnarray*}
Here the 11 different real coefficients $b_{jk},c_{jk},x_1,x_2,y_1,\rho,\theta_4$ are arbitrary and $\theta_1,\theta_2,\theta_3$ are given by
\begin{equation*}
\begin{pmatrix}
\theta_1\\
\theta_2\\
\theta_3
\end{pmatrix}
=\frac{1}{2}\begin{pmatrix}
-{\rho}c\ss12 - b\ss11c\ss21 + b\ss21c\ss11\\
-{\rho}c\ss11 - b\ss11c\ss22 + b\ss21c\ss12\\
{\rho}b\ss11 - c\ss11c\ss22 + c\ss12c\ss21
\end{pmatrix}.
\end{equation*}
\end{proposition}

\begin{proof}
The arguments needed here are exactly the same as already provided in the proof of Proposition \ref{proposition-SU2*SU2}, for the details see \cite{Tur}.
\end{proof}

\begin{theorem}\label{theorem-SLR2*SO2}
Let $G$ be a six-dimensional Riemannian Lie group containing the subgroup $\SLR2\times\SO2\,.$ And let $\V$ be the left-invariant distribution generated by the Lie algebra $\slr2\times\so2$. The resulting foliation $\F$ tangent to $\V$ is
\begin{enumerate}
\item conformal if and only if $y_1+x_2=0$ and $x_1=0$,
\item Riemannian if and only if it is conformal,
\item minimal if and only if $\rho=0$, and
\item totally geodesic if and only if $$\rho = b_{11}=b_{21}= c_{11}=c_{12}=c_{21}=c_{22} = 0\,,$$
\end{enumerate}
where the parameters $b_{11}\,, b_{21}\,, c_{11}, c_{12}\,, c_{21}\,, c_{22}\,, x_1\,, x_2\,, y_1\,,$ and $\rho$ are defined in Proposition \ref{proposition-SLR2*SO2}.
\end{theorem}

\begin{proof}
We start by determining when the foliation $\F$ is conformal. Calculating the second fundamental form of the horizontal distribution $\H$ we find that 
$$\BH(XX) = x_1\, T,\ \ \BH(YY) = -x_1\, T,$$
$$\BH(XY) = \frac{1}{2}(y_1+x_2)\, T.$$
From this it is clear that $\F$ is conformal exactly when $y_1+x_2=0$ and $x_1 = 0$.  This is also equivalent to $\F$ being Riemannian.

Evaluating the second fundamental form $B^\V$ for the basis elements in the vertical distribution $\V$ we observe that 
$$\BV(AA) = 0,\ \ \BV(AB) = -b\ss11\,X - b\ss21\,Y,$$
$$\BV(AC) = -c\ss11\,X - c\ss21\,Y,\ \ 
\BV(AT) = \frac{y_1}{4}(c\ss22\,X - c\ss12\,Y),$$
$$\BV(BB) = 0,\ \ \BV(BC) = 0,\ \ 
\BV(BT) = \frac{y_1}{4}(c\ss21\,X - c\ss11\,Y),$$
$$\BV(CC) = 0,\ \ \BV(CT) = \frac{y_1}{4}(-b\ss21\,X + b\ss11\,Y),\ \ \BV(TT) = \rho\,Y.$$
It is immediately apparent that $\F$ is minimal if and only if $\rho=0$. In addition to this, we find that $\F$ is totally geodesic exactly when $$\rho = b_{11} = b_{21} =c_{11} = c_{21} = y_1c_{22} = y_1c_{12} = 0.$$
\end{proof}

Proposition \ref{proposition-SLR2*SO2} and Theorem \ref{theorem-SLR2*SO2} provide a new $8$-dimensional family of $6$-dimensional Lie groups carrying a conformal foliation with minimal leaves of codimension two. This is generated by the $8$ parameters $b_{11}$,  $b_{21}$, $c_{11}$, $c_{21}$, $c_{12}$, $c_{22}$, $y_1$, $\theta_4$ which are defined in the statement of Proposition \ref{proposition-SLR2*SO2}.

\section{Acknowledgements}

The authors would like to thank the referee for useful comments on the presentation.
\vskip .2cm
The authors would like to thank Victor Ottosson for providing us with his Maple programmes from \cite{Ott}.  These have turned out to be very useful for checking our own calculations.
\vskip .2cm
The first author would like to thank the Department of Mathematics at Lund University for its great hospitality during her time there as a postdoc.

\end{document}